\documentclass[12pt]{article}
\usepackage{e-jc}

\usepackage{amsthm,amsmath,amssymb}

\usepackage{graphicx}

\usepackage[colorlinks=true,citecolor=black,linkcolor=black,urlcolor=blue]{hyperref}

\theoremstyle{plain}
\newtheorem{theorem}{Theorem}
\newtheorem{lemma}[theorem]{Lemma}

\theoremstyle{definition}

\newtheorem{conjecture}[theorem]{Conjecture}

\theoremstyle{remark}

\title{\bf King-serf duo by monochromatic paths\\ in k-edge-coloured tournaments}

\author{Krist\'of B\'erczi \qquad Attila Jo\'o\\
\small MTA-ELTE Egerv\'ary Research Group\\[-0.8ex] 
\small Budapest, Hungary\\
\small {\tt berkri@cs.elte.hu joapaat@cs.elte.hu} \\
}

\date{\dateline{July 14, 2016}{March 3, 2017}}

\begin{document}

\maketitle
The paper has been published in the Electronic Journal of Combinatorics (see \cite{berczi2016king})

\begin{abstract}
An open conjecture of Erd\H{o}s states that for every positive integer $k$ there is a (least) positive integer $f(k)$ so that whenever a 
tournament has its edges colored with $k$ colors, there exists a set $S$ of at most $f(k)$ vertices so that every vertex has a 
monochromatic path to some point in $S$. We consider a related question and show  that for every (finite or infinite) cardinal $\kappa>0$ there is a cardinal $ \lambda_\kappa  $ such that in every $\kappa$-edge-coloured tournament there exist  disjoint vertex  sets $K,S$ with total size at most $ \lambda_\kappa$ so that every vertex $ v $ has a monochromatic path of length at most two from $K$ to $v$ or from $v$ to $S$. 

\end{abstract}

\section{Introduction}

A \textbf{tournament} $ T=(V(T), A(T)) $ is a directed graph obtained by orienting the edge set of a (possibly infinite) complete 
undirected graph. A directed cycle is called a \textbf{dicycle} for short. We use some basic set theoretic conventions.  We consider functions $ f $ as sets of ordered pairs where $ \left\langle x,y  \right\rangle\in f $ and $\left\langle x,z  \right\rangle\in f   $ imply $ y=z $. For a finite or infinite cardinal $ \kappa $ let $ \mathsf{exp}_0(\kappa)=  \kappa $ and let $\boldsymbol{\mathsf{exp}_{k+1}(\kappa)}=2^{\mathsf{exp}_{k}(\kappa)} $. Remember that a cardinal is the set of the ordinals that are smaller than itself, for example $ 3=\{ 0,1,2 \} $.  A $ \boldsymbol{\kappa} $\textbf{-edge-colouring} of a tournament $T$ is a function $c:A(T)\rightarrow \kappa$.
A \textbf{monochromatic path} is a directed path (repetition of vertices is not allowed) with edges having the same colour.   We call a dicycle \textbf{quasi-monochromatic} if all but at most one of its edges have the same colour. 

Our investigation was motivated by the following conjecture of Erd\H{o}s \cite[p. 274]{sands1982monochromatic}.

\begin{conjecture}[Erd\H{o}s] \label{conj:erdos}
For every positive integer $k$ there is a (least) positive integer $f(k)$ so that every $k$-edge-coloured finite tournament admits  a subset $S\subseteq V(T)$ of size at most $f(k)$ such that $S$ is reachable from every vertex by a monochromatic path.
\end{conjecture}

It is known that $f(1)=f(2)=1$, and there is an example showing that $f(3)\geq 3$ (see \cite{sands1982monochromatic}).  However, there 
is no known constant upper bound for $f(3)$, although it is conjectured to be $3$ by Erd\H{o}s. As a weakening of the original conjecture, we 
consider source-sink pairs instead of one sink set $ S $. However, we may add bounds on the length of the monochromatic paths. More 
precisely, a \textbf{king-serf duo by monochromatic paths} consists of disjoint vertex sets $K,S\subseteq V(D)$ so that every vertex $ v 
$ has a monochromatic path of length at most two from $K$ to $v$ or from $v$ to $S$. The \textbf{size} of the duo is defined as 
$|K|+|S|$. An edge $ uv $ of an edge-coloured tournament $T$ is called \textbf{forbidding} if there is no monochromatic path of length at 
most two from $v$ to $ u $.  Note that if $ T' $ is a subtournament of $ T $ containing a forbidding edge $ uv $, then $ uv $ is 
forbidding edge with respect to $ T'  $ as well. 

The main result of the paper is the following.

\begin{theorem}\label{thm:main}
For every (finite or infinite) cardinal $ \kappa $ there is a cardinal $ \lambda_\kappa \leq \mathsf{exp}_{10}(\kappa) $ such that in every $\kappa$-edge-coloured tournament there exists a king-serf duo by monochromatic paths of size at most $ \lambda_k $. For finite $ \kappa $ one can guarantee $ \lambda_\kappa \leq \kappa^{62500\kappa} $.
\end{theorem}

The rest of the paper is organized as follows. In Section~\ref{sec:previous}, we give an overview of previous results. Theorem~\ref{thm:main} is then proved in Section~\ref{sec:main}.

\section{Previous work} \label{sec:previous}

Given a digraph $D=(V,A)$, an independent set $K\subseteq V$ is called a kernel if it is absorbing, that is,
there exists a directed edge from $K$ to $v$ for every $v\in V-K$. Kernels were introduced by Von Neumann and Morgenstern \cite{neumann44} in relation to game theory.

The concept of kernels was generalized by Galeana-S\'anchez \cite{galeana96} for edge-coloured digraphs. In the coloured case, independence and absorbency are only required by means of monochromatic paths, hence these sets are called kernels by monochromatic paths. The existence of such kernels is widely studied, see \cite{galeana98}-\cite{galeana09b}, \cite{minggang88}. The case when $K$ is an absorbing set but not necessarily independent by monochromatic paths is also of interest. Since an absorbing set always exists in a $k$-coloured digraph, a natural problem is to find one with minimum size, which motivates the conjecture of Erd\H{o}s (Conjecture~\ref{conj:erdos}). In \cite{sands1982monochromatic}, Sands, Sauer and Woodrow proved that every $2$-edge-coloured tournament admits an absorbing vertex, and also presented a $3$-edge-coloured tournament in which the minimum size of an absorbing set is $3$. They conjectured that every $3$-edge-coloured tournament without polychromatic dicycles of length $3$ has an absorbing vertex. Minggang \cite{minggang88} verified a slightly different version of the conjecture claiming that any $k$-edge-coloured tournament without polychromatic -not necessarily directed- cycles of length $3$ contains an absorbing vertex. Meanwhile, examples show that for every $k\geq 5$, there exists a $k$-edge-coloured tournament without polychromatic dicycle of length $3$ without an absorbing vertex. Galeana-S\'anchez \cite{galeana96} proved that if each directed cycle of length at most 4 in a $k$-edge-coloured tournament $T$ is quasi-monochromatic then $T$ has an absorbing vertex. In his PhD thesis \cite{bland11}, Bland provided several sufficient conditions for the existence of an absorbing vertex in a $k$-edge-coloured tournament. He also gave a sufficient condition for the existence of an absorbing set of size $3$ in $3$-edge-coloured tournaments.

Quasi-kernels are possible weakenings of kernels. An independent set $K\subseteq V$ is a quasi-kernel if for each vertex $v\in V-K$ there exists a path of length at most $2$ from $K$ to $v$ (quasi-sink sets can be defined analogously). The fundamental theorem of Chv\'atal and Lov\'asz \cite{lovasz74} shows that every finite digraph contains a quasi-kernel. In \cite{soukup09}, P.L. Erd\H{o}s and Soukup studied the existence of quasi-kernels in infinite digraphs. As the plain generalization of the Chv\'atal-Lov\'asz theorem fails even for tournaments, they considered the problem of finding a partition $V=V_1\cup V_2$ of the vertex set such that the induced subgraph $D[V_1]$ has a quasi-kernel and $D[V_2]$ has a quasi-sink. The authors conjectured that such a partition exists for any (possibly infinite) digraph. They verified that every (possibly infinite) directed graph $D=(V,A)$ contains two disjoint, independent subsets $K$ and $S$ of $V$ such that for each node $v\in V$ there exists a path of length at most $2$ from $K$ to $v$ or from $v$ to $S$, but the conjecture is still open.

The motivation of our investigations was to combine the notions of absorbing sets by monochromatic paths and that of quasi-kernels and sinks, which lead to the definition of a king-serf duo by monochromatic paths, and to prove an analogue of Conjecture~\ref{conj:erdos}.


\section{Proof of Theorem~\ref{thm:main}} \label{sec:main}

The proof relies on the following theorem due to Erd\H{o}s, Hajnal and P\'osa \cite{erdos1975strong} (finite case) and Hajnal \cite{hajnal1991embedding} (infinite case).
 
\begin{theorem}[Erd\H{o}s, Hajnal and P\'osa] \label{thm:hajnal}
For every finite simple graph $ H $ and cardinal $ \kappa>0  $ there is a simple graph $ G $ of size at most  $ \mathsf{exp}_{\left|V(H)\right|+5}(\kappa) $ (at most $ \kappa^{500\left|V(H)\right|^{3}\kappa} $ in the finite case) such that in any $ \kappa $-edge-colouring of $ G $ one can find a monochromatic induced subgraph isomorphic to $ H $.
\end{theorem}

With the help of Theorem~\ref{thm:hajnal}, first we prove the following.

\begin{lemma}\label{lem:quasi}
For every cardinal $ \kappa>0 $ there exists a tournament $ T_\kappa $ of size at most $\mathsf{exp}_{10}(\kappa) $ (at most $ \kappa^{62500\kappa} $ in the finite case) such that in any $ \kappa $-edge-colouring of $ T_\kappa $ there exists a quasi-monochromatic dicycle of length three. 
\end{lemma}
\begin{proof}
Pick a graph $G$ ensured by Theorem~\ref{thm:hajnal} for $\kappa$ and $H=C_5$, that is, a cycle of length $5$. Fix a well-ordering of $V(G)$. Let $T_\kappa$ denote the tournament obtained by orienting the edges of $G$ forward according to the ordering, and by adding all missing edges as backward edges. We claim that $T_\kappa$ satisfies the conditions of the lemma.

Take an arbitrary $ \kappa $-edge-colouring of $ T_\kappa$. The choice of $G$ implies that there is a monochromatic (not necessarily directed) cycle $C$ of length $5$ in the graph such that $A(C)$ consists of forward edges, and all the other edges induced by $ V(C) $ in $ T_\kappa $ are backward edges.

No matter how the edges of $C$ are oriented, we can always find a directed path of length two in $A(C)$. Take such a path, say $uv$ and $vw$. These edges together with $wu$ form a quasi-monochromatic dicycle, concluding the proof of the lemma. 
\end{proof}

We claim that $\lambda_\kappa := |V(T_\kappa)| $ satisfies the conditions of the theorem. Suppose to the contrary that there exists a $ \kappa $-edge-coloured tournament $T$ not containing a king-serf duo by monochromatic paths of size at most $ \lambda_\kappa $. Let $ T_\kappa $ be a tournament that we obtain by applying Lemma~\ref{lem:quasi}.

\begin{lemma}\label{lem:forbidden}
$T$ has a subtournament isomorphic to $ T_\kappa $ consisting of forbidding edges.
\end{lemma}
\begin{proof}
We build up the desired subtournament by transfinite recursion. Let $ V(T_\kappa)=\{ u_\gamma \}_{\gamma< \left|V(T_\kappa)\right|} $. Assume that for some $ \alpha<\left|V(T_\kappa)\right| $ we have already found an $\subset$-increasing chain $ \left\langle f_\beta: \beta<\alpha  \right\rangle $ of $ T_\kappa \rightarrow T $ embeddings where 
$ \mathsf{dom}(f_\beta)=\{ u_\gamma \}_{\gamma<\beta} $ and the images of the edges of $ T_\kappa $ are forbidding edges of $ T $. If $ 
\beta $ is a limit ordinal, we may simply take $ f_\beta:= \bigcup_{\gamma<\beta}f_\gamma $ to keep the conditions. Assume that $ 
\beta=\delta+1 $. Let $ O =\{\gamma<\delta:  u_\delta u_\gamma\in A(T_\kappa)  \} $. 
As $T$ is a counterexample, the sets $ K:= \{ f_\delta(u_\gamma) \}_{\gamma\in O} $ and $ S:= \{ f_\delta(u_\gamma) \}_{\gamma \in \delta 
\setminus O} $ cannot form a king-serf duo by monochromatic paths. Therefore there is a vertex $ v\in V(T) $ such that there is a 
forbidding edge from $v$ to every element of $ K $, and there is a forbidding edge from every element of $ S $ to $v$. But then $ 
f_{\delta+1}:= f_\delta\cup \{ \left\langle u_\delta,v  \right\rangle \} $ maintains the conditions. Finally, the image of $ 
f:=\bigcup_{\gamma<\left|V(T_\kappa)\right|} $ gives the desired copy of $ T_\kappa $.
\end{proof}

The $\kappa$-edge-colouring of $T$ defines a $ \kappa $-edge-colouring of its $ T_\kappa $ subgraph as well. Therefore, by the choice of $T_\kappa$, there is a quasi-monochromatic dicycle $ C $ of length three in $ T_\kappa $. Let $ uv $ denote the edge of $ C $ with different colour than the others if $C$ contains two colours, and let $uv$ be an arbitrary edge of $C$ if it is monochromatic. Then $ C-uv $ is a monochromatic path of length two from $ v $ to $ u $, contradicting $ uv $ being a forbidding edge of $ T$. This finishes the proof of Theorem~\ref{thm:main}.

\subsection*{Acknowledgements}

The authors are supported by the Hungarian National Research, Development and Innovation Office -- NKFIH
grants K109240

\bibliographystyle{plain} 

\end{document}